\documentclass[12pt]{amsart}
\usepackage[pagewise]{lineno}
\usepackage{lipsum}
\usepackage{fullpage}
\usepackage{graphicx} 

\usepackage{cases}
\usepackage{float}
\usepackage{hyperref}
\usepackage{blindtext}
\usepackage{multicol}
\usepackage{graphics}
\usepackage{graphicx}
\usepackage{epstopdf}
\usepackage{amssymb}
\usepackage{amsmath}
\usepackage{amsthm}
\usepackage{array}
\usepackage{latexsym}
\usepackage{amsfonts}
\usepackage{color}
\usepackage{verbatim}
\usepackage[numbers]{natbib}
\usepackage{amsthm,url,xspace}
\definecolor{couleurCitations}{rgb}{0,0,0.85}
\definecolor{couleurRef}{rgb}{0.75,0,0}

\newtheorem{theorem}{Theorem}

\newtheorem{lema}{Lemma}

\newtheorem{prop}{Proposition}

\theoremstyle{definition}


\title{Hyperbolic knots with a large toroidal surgery}
\author{Mario Eudave-Mu\~noz and Masakazu Teragaito}
\address{Instituto de Matem\'aticas, Universidad Nacional Aut\'onoma de M\'exico, Campus Cuernavaca, Morelos, M\'exico}
\email{mario@matem.unam.mx}

\address{International Institute for Sustainability with Knotted Chiral Meta Matter (WPI-SKCM$^2$), Hiroshima University, 
1-3-1 Kagamiyama, Higashi-Hiroshima, 739--8526, Japan}
\email{teragai@hiroshima-u.ac.jp}
\thanks{The first author has been supported by PAPIIT UNAM grant IN117423. The second author  has been supported by
JSPS KAKENHI Grant Number 20K03587. }

\date{September 2023}

\begin{document}

\maketitle

\begin{abstract}
We show an infinite family of hyperbolic knots that have an exceptional surgery producing a graph manifold containing five disjoint, and non parallel incompressible tori.\end{abstract}

\section{Introduction}

Let $K$ be a hyperbolic knot in the 3-sphere. By the Hyperbolic Dehn Surgery Theorem \cite{Thurston}, it is well known that most surgeries on $K$ produce hyperbolic manifolds, except for a finite number of slopes. Among the non-nyperbolic surgeries are the toroidal surgeries, that is, surgeries producing a toroidal manifold, which is just a manifold containing an incompressible torus. 
It is known that if $p/q$-surgery on $K$ produces a toroidal manifold then $\vert q \vert \leq 2$ \cite{GordonLuecke1} \cite{GordonLuecke2}, and if $\vert q \vert=2$, then all knots having this kind of surgery have been determined \cite{EM1}, \cite{GordonLuecke3}. For the case $\vert q \vert=1$, many examples of knots with this type of surgery are known, see \cite{EM-RL} for a survey on this topic, however it seems difficult to have a determination of all such knots. Most of the known examples produce a manifold containing a unique essential torus, and in many cases the manifold produced is a graph manifold. In \cite{EM-RL} examples are given of knots having a surgery that produce a manifold containing two non-isotopic disjoint incompressible tori. According to Y. Tsutsumi \cite{Tsutsumi}, K. Motegi asked if there exists a universal upper bound for the number of JSJ pieces that a manifold obtained by Dehn-surgery on a hyperbolic knot in $S^3$ can have. 
Another way of formulating this question is asking if there is a universal upper bound on the number of disjoint non-parallel incompressible tori that can appear after performing Dehn surgery on a hyperbolic knot in $S^3$. Any such a torus comes from a punctured torus properly embedded in the knot exterior, which fills up to a closed surface after performing Dehn surgery along the slope given by the punctured torus. Tsutsumi \cite{Tsutsumi} considered the case of once punctured tori lying in a knot exterior, that is, of
Seifert surfaces, and showed that in the exterior of a hyperbolic knot there are at most seven disjoint non-parallel incompressible genus one Seifert surfaces. This bound was improved to at most five Seifert surfaces by L. Valdez-S\'anchez \cite{Valdez}, who also showed an 
infinite family of knots which realize this bound. 

In this note we show two infinite families of hyperbolic knots, the knots in one family have a surgery producing a graph manifold which contains five disjoint non-parallel incompressible tori, and the knots in the other family have a surgery producing a graph manifold which contains four disjoint non-parallel incompressible tori. All of the incompressible tori come from twice punctured tori properly embedded in the knot exterior. The examples are constructed by means of the Montesinos trick \cite{Mon}. 
Recently, R. Aranda-Cuevas, E. Ram\'{\i}rez-Losada and J. Rodr\'{\i}guez-Viorato \cite{ARR} have shown that in a hyperbolic knot exterior there are at most six disjoint, non-parallel, incompressible, nested and properly embedded twice punctured tori.  This result complements our examples, and both give evidence that the optimal bound for the number of tori must be five, considering any kind of torus.

\section{The examples}

The construction of the examples will be done by means of the Montesinos trick. So, we start with a 2-string tangle which has a filling producing a trivial knot and another filling producing a knot or link which contains several disjoint Conway spheres. After taking double branched covers, the resulting knot will have the required properties.

Let $B(\ell,m,n,p,q)$ be the tangle shown in Figure \ref{ovillo}, where each box represents a collection of horizontal crossings given by the integral parameters $\ell,m,n,p,q$, with the convention shown in Figure \ref{ovillo}.
We assume that the parameters satisfy that $\vert \ell \vert \geq 2$, $\vert m \vert \geq 2$, $\vert p \vert \geq 2$, $\vert n \vert \geq 3$ and $\vert q \vert \geq 1$, but $(p,q) \not= \pm (2,1)$.

\begin{figure}
    \centering
        \includegraphics[scale=.5]{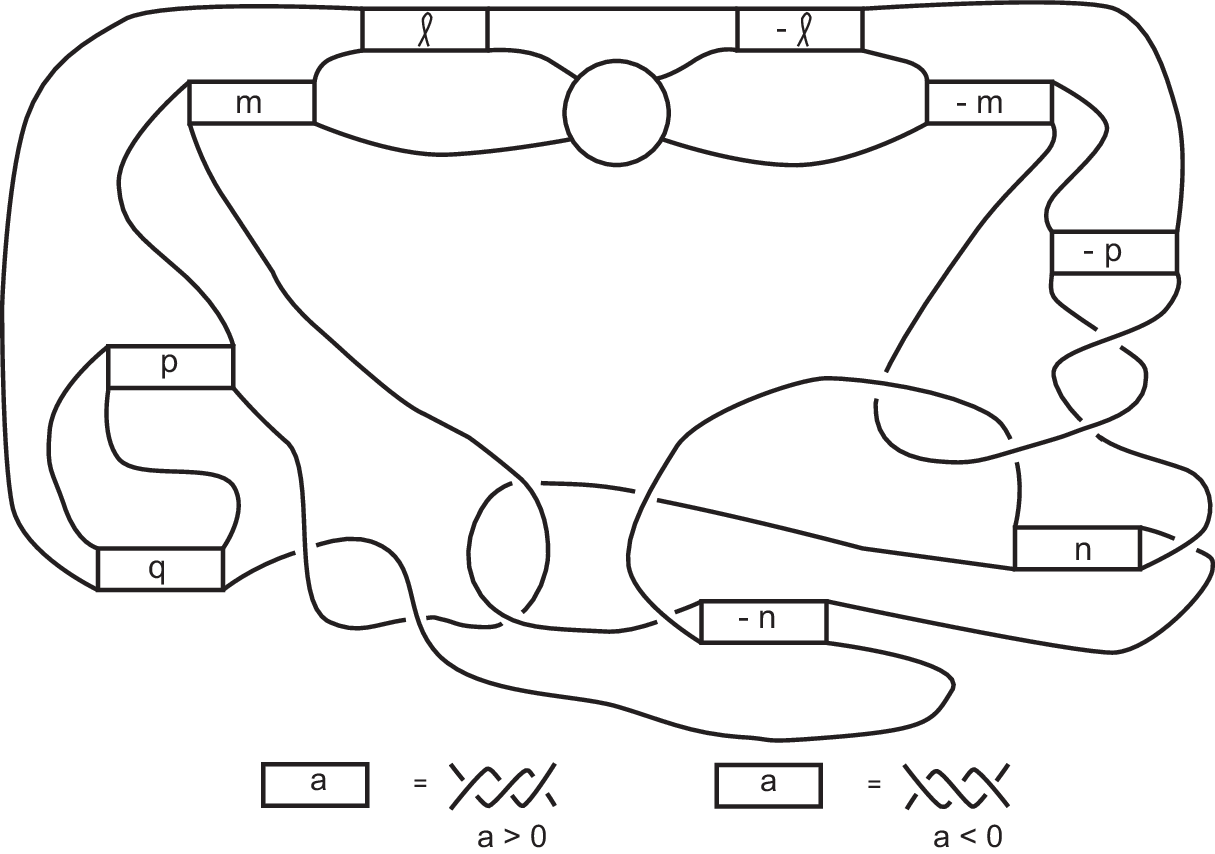}
         \caption{The tangle $B(\ell,m,n,p,q)$}
    \label{ovillo}
        \end{figure}

\begin{theorem} The tangle $B(\ell,m,n,p,q)$ has the following properties:
\begin{enumerate}

    \item $B(\ell,m,n,p,q)\cup R(0)$ is $(S^3,U)$, where $U$ is the trivial knot. 
    \item $B(\ell,m,n,p,q)\cup R(1/0)$ is 
    $(S^3,k)$, where $k$ is a knot or link that has five disjoint and non-isotopic Conway spheres.

\end{enumerate}
Here, $R(0)$ and $R(1/0)$ are ordinary rational tangles with slope $0$ and $1/0$, respectively.
\end{theorem}

\begin{proof} Item (1) follows from the sequence of pictures in Figure \ref{trivialknot}. Item (2) follows from Figure \ref{ConwaySpheres}. The tangles obtained in the decomposition are shown in Figure \ref{pedazosJSJ}. Note that no one is a trivial tangle. \end{proof}

\begin{figure}
    \centering
        \includegraphics[width=12cm]{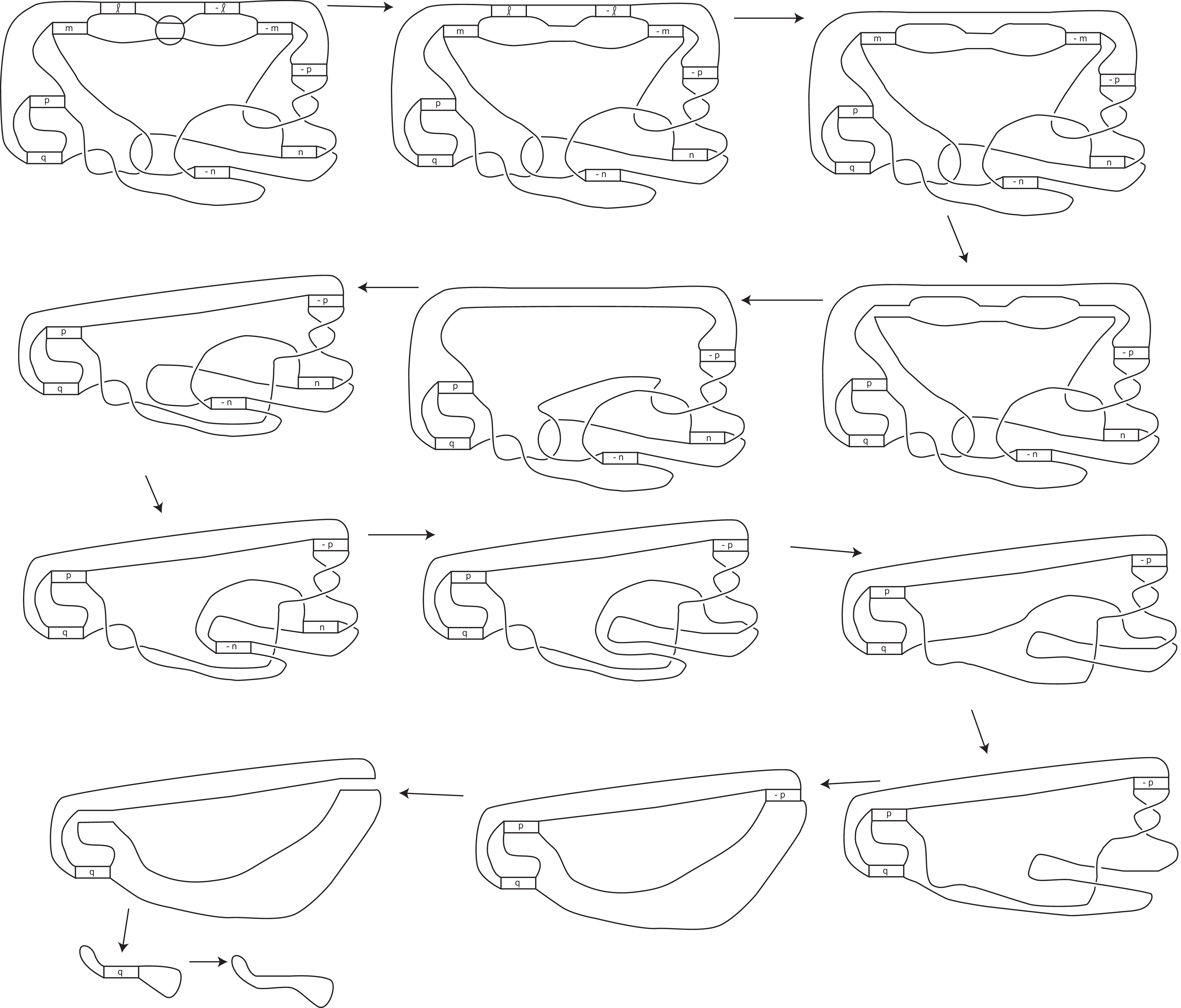}
        \caption{The trivial knot $(S^3,U)=B(\ell,m,n,p,q)\cup R(0)$}
        \label{trivialknot}
        \end{figure}

\begin{figure}
    \centering
        \includegraphics[width=10cm]{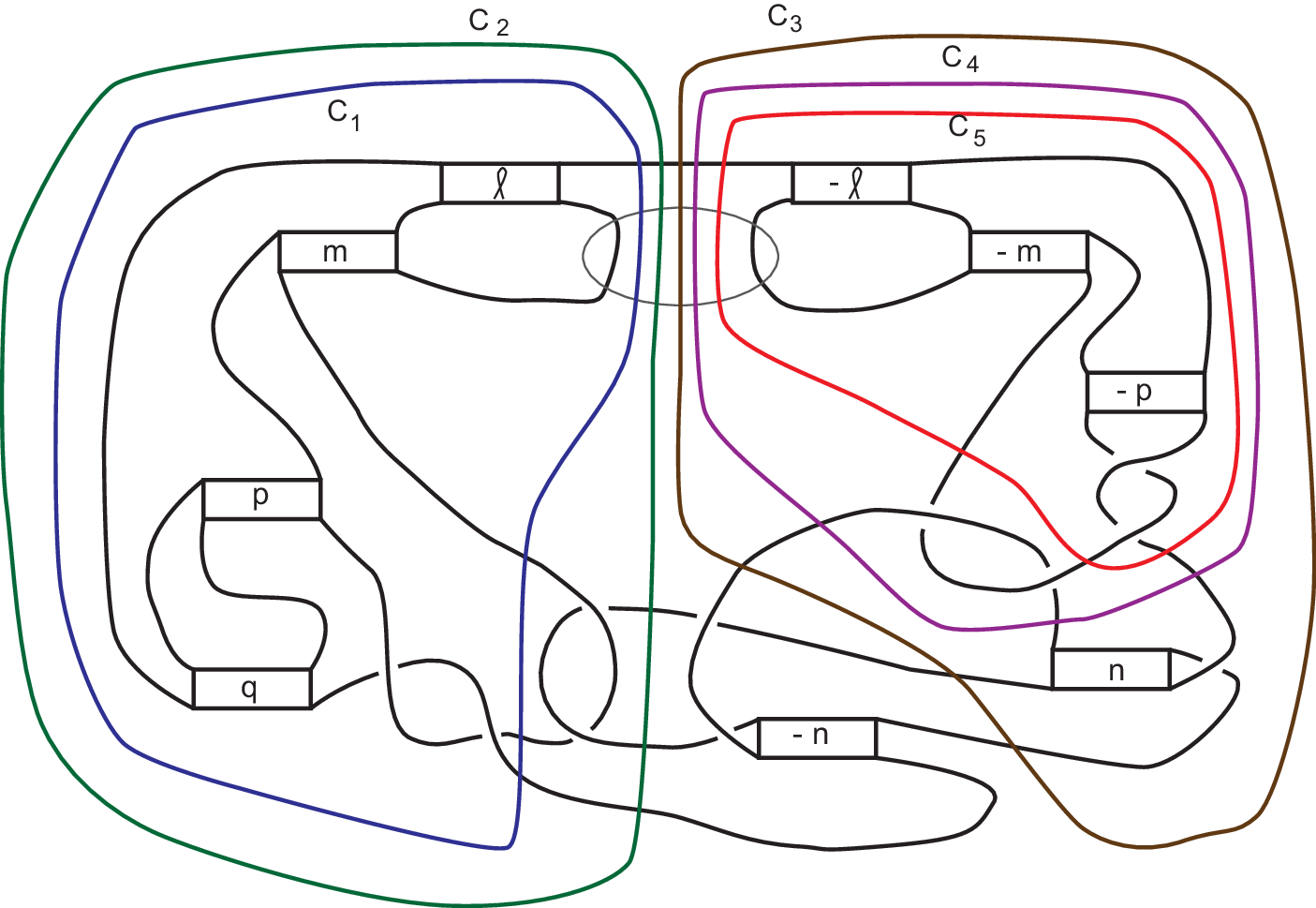}
        \caption{The knot $(S^3,k)=B(\ell,m,n,p,q)\cup R(1/0)$}
        \label{ConwaySpheres}
        \end{figure}

\begin{figure}
    \centering
        \includegraphics[width=12cm]{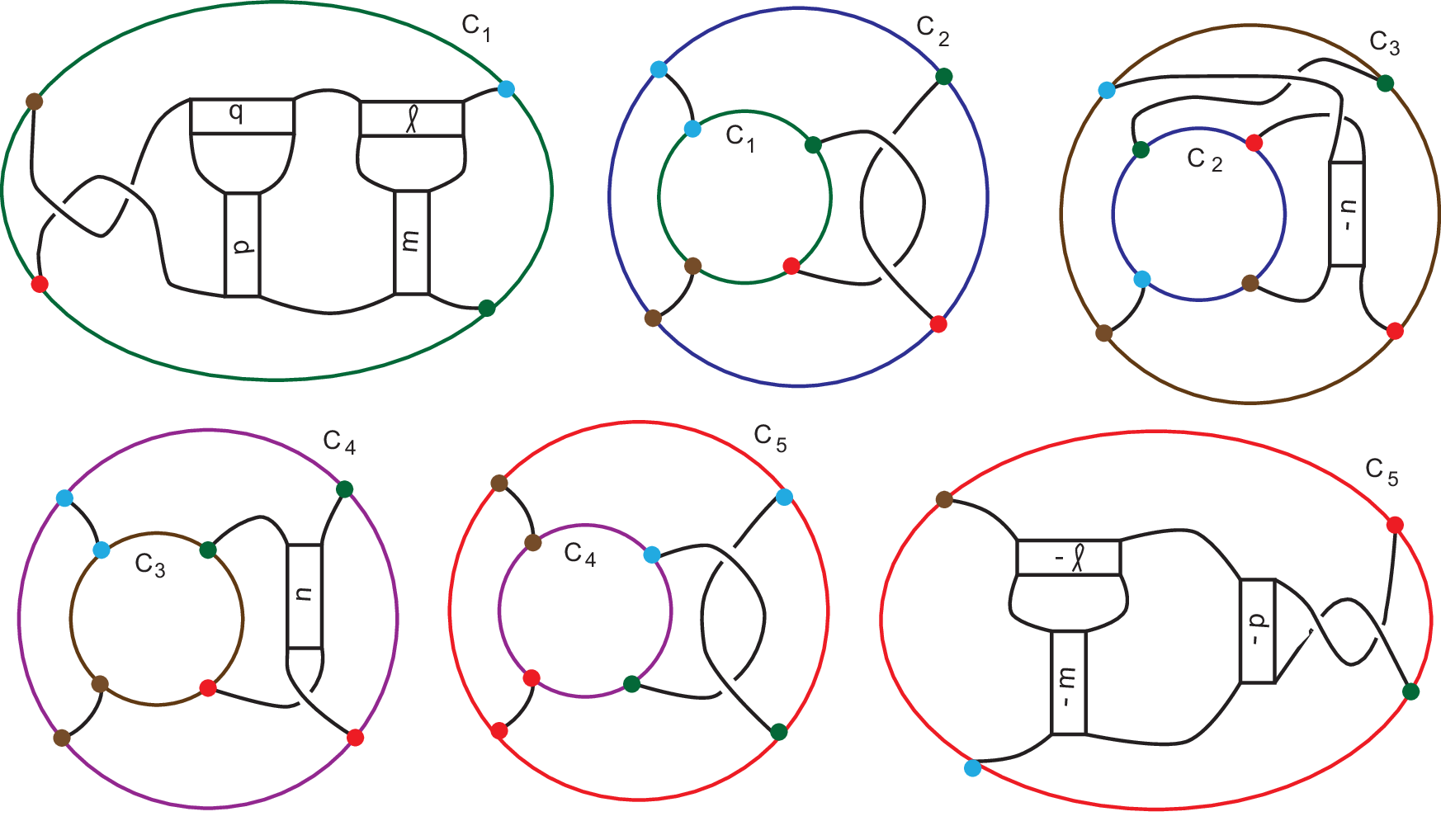}
        \caption{Tangle decomposition of $B(\ell,m,n,p,q)\cup R(1/0)$}
        \label{pedazosJSJ}
        \end{figure}

We just showed that $B(\ell,m,n,p,q)$ is a trivializable tangle, that is, it can be filled with a rational tangle to produce the trivial knot. Let $\pi\colon S^3 \rightarrow S^3$ be the double cover of $S^3$ branched along $U$, which is the 3-sphere because $U$ is the trivial knot. Let $N=\pi^{-1}(R(0))$ be the lift of the tangle $R(0)$, this is a solid torus embedded in $S^3$, and then its core is a knot, called the covering knot of the tangle $B(\ell,m,n,p,q)$, and which is denoted by $\tilde K = \tilde K(\ell,m,n,p,q)$. By the Montesinos trick, the double cover of $S^3$ branched along $k$ is obtained by $r$-Dehn surgery on $\tilde K$ for some integral slope $r$; this manifold is denoted by $\tilde K(r)$. 

Let $C_1$, $C_2$, $C_3$, $C_4$, $C_5$ be the five Conway spheres of the knot or link $k$, as shown in Figure \ref{ConwaySpheres}. Let $T_i=\pi_r^{-1}(C_i)$, $1\leq i\leq 5$,
where $\pi_r\colon \tilde K(r) \rightarrow S^3$ is the double cover of $S^3$ branched along $k$. Each $T_i$ is a separating incompressible torus in $\tilde K(r)$. 

\begin{prop} $\tilde K(r)$ is a graph manifold, determined by a linear graph consisting of 6 vertices and 5 edges, each extremal vertex represents a Seifert fiber space over a disk $D^2$, and the internal vertices represent Seifert fibered spaces over an annulus $A^2$. The Seifert invariants of the pieces are given by $(D^2;{\frac{2pq-2+q}{pq-1}},{\frac{\ell}{m\ell-1}})$, $(A^2;\frac{1}{2})$, $(A^2;\frac{1}{n+1})$, $(A^2;\frac{-1}{n+1})$, $(A^2;\frac{1}{2})$, $(D^2;{\frac{\ell}{1-m\ell},\frac{2p-1}{p}})$.
In particular, $\tilde K(r)$ is irreducible, and any incompressible torus is isotopic to one of $T_i$.
\end{prop}

\begin{proof} This follows directly from Figures \ref{ConwaySpheres} and \ref{pedazosJSJ}, just calculating the rational tangles given in each of the tangles.
(Our convention has the opposite sign of \cite{HM}.)
In fact, this gives the JSJ decomposition of $\tilde K(r)$. We can see from Figure \ref{pedazosJSJ} that the Seifert fibrations of adjacent pieces do not match.
\end{proof}

Let $k_r$ be the core of the solid torus $\pi_r^{-1}(R(1/0))$. This is a knot in the manifold $\tilde K(r)$, and it is the lift of an arc in $R(1/0)$ connecting the two strands of the tangle. Note that $k_r$ intersects each torus $T_i$ in exactly two points.

\begin{lema} \label{intersecting} The knot $k_r$ cannot be made disjoint from any of the tori $T_i$. \end{lema}

\begin{proof} Note that $k_r$ is the lift of an arc in $R(1/0)$ connecting the two strands of the tangle (an arc with no crossing in the diagram of $R(1/0)$). Suppose that $k_r$ can be isotoped to be disjoint from one of the tori, then it must be disjoint from one of the tori, $T_1$ and $T_5$, in the ends, say from $T_1$. Let $H_1$, $H_2$, $H_3$, $H_4$, $H_5$, $H_6$ be the pieces in the $JSJ$ decomposition of $\tilde K(r)$, where $H_1$ is bounded by $T_1$, $H_i$ is bounded by $T_{i-1}$ and $T_i$, $i=2,3,4,5$, and $H_6$ is bounded by $T_5$. Note that the intersection of $k_r$ with each of $H_i$, for $i=2,3,4,5$, consists of two arcs joining different boundary components of $H_i$, and that the intersection of $k_r$ with $H_1$ (and $H_6$) consists of an arc which is an unknotting tunnel for $H_1$ ($H_6$). Suppose then that $T_1$ can be isotoped to a torus $T_1'$ disjoint from $k_r$. We proceed as in \cite[Section 1]{EM2}. $T_1'$ intersects the tori $T_i$ in a collection of simple closed curves. There must be a product region, cobounded by a subsurface of $T_1'$ and one of $T_i$. An arc of $k_r$ cannot lie in this product region, for these are arcs connect two different tori or are unknotting tunnels of a region. Then the torus $T_1'$ can be isotoped to eliminate some intersections. Repeating the operation we get that $T_1'$ is disjoint from the tori $T_i$, and it must cobound a product region with $T_1$, but an arc of $k_r$ cannot lie in the product region. Then $k_r$ cannot be made disjoint from any of the tori $T_i$.
\end{proof}

\begin{theorem} \label{main}
\begin{enumerate}
\item $\tilde K(\ell,m,n,p,q)$ is a hyperbolic knot. 

\item For a certain slope $r$, 
$\tilde K(\ell,m,n,p,q)(r)$  contains five disjoint non-isotopic incompressible tori.

\end{enumerate}
\end{theorem}

\begin{proof} First note that $\tilde K$ is not a torus knot, for it has a surgery producing manifolds which contain separating incompressible tori. Suppose that $\tilde K$ is a satellite knot, and let $T$ be an incompressible non-boundary parallel torus in the exterior of $\tilde K$. Suppose that $T$ bounds a solid torus $V$ in which $\tilde K$ lies. We can also suppose that the exterior of $V$ does not contain any other incompressible torus. If $T$ remains incompressible after $r$-Dehn surgery, then it cannot be isotopic to one of the tori $T_i$, for these tori can not be made disjoint from the core of the surgered solid torus by Lemma \ref{intersecting}. Then $T$ must compress in $\tilde K(r)$. By \cite{Gabai,Sch}, $\tilde K$ is a 0- or 1-bridge braid in $V$ of winding number $\omega \geq 2$, and after $r$-Dehn surgery on $\tilde K$, $V$ becomes a solid torus $V'$. That shows that $\tilde K(r)$ is also obtained by surgery along the core of $V$.
In particular, this core is not a torus knot again.
By \cite{Gordon}, it is obtained by $r/\omega^2$-surgery on $V$, but as $\omega^2 \geq 4$ and the core of $V$ is an hyperbolic knot, this is not possible by \cite{GordonLuecke1}. \end{proof}

Let $Q(a,b,c,d,e,f)$ be the tangle shown in Figure \ref{tangle}, where each box represents a collection of horizontal crossings given by the integral parameters $a,b,c,d,e,f$, with the same convention as in Figure \ref{ovillo}. We assume that the parameters satisfy that $\vert a \vert \geq 2$, $\vert b \vert \geq 1$, $\vert c \vert \geq 2$, $\vert d \vert \geq 3$, $\vert e \vert \geq 2 $, and $\vert f \vert \geq 1$, but $(a,b),(e,f)\ne \pm (2,1)$.

\begin{figure}
    \centering
        \includegraphics[scale=.5]{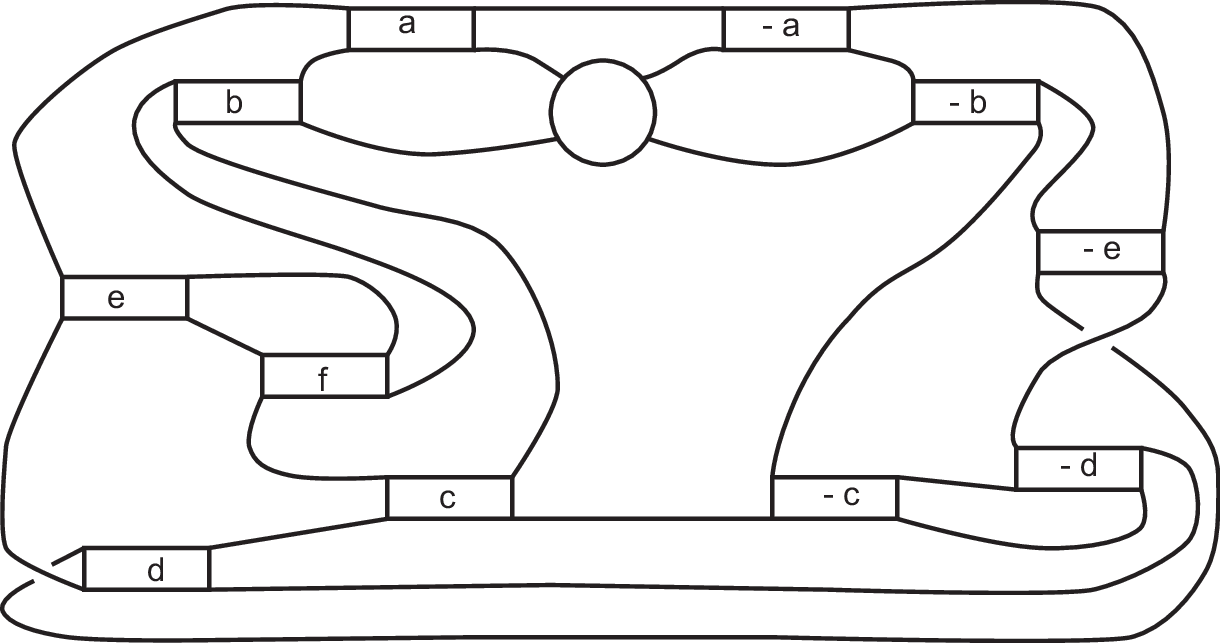}
         \caption{The tangle $Q(a,b,c,d,e,f)$}
    \label{tangle}
        \end{figure}

 \begin{theorem} The tangle $Q(a,b,c,d,e,f)$ has the following properties:
\begin{enumerate}

    \item $Q(a,b,c,d,e,f)\cup R(0)=(S^3,U)$, where $U$ is the trivial knot.
    \item $Q(a,b,c,d,e,f)\cup R(1/0)$ gives a knot or link that has four disjoint and non-isotopic Conway spheres, denoted by $S_1$, $S_2$, $S_3$, $S_4$.
  
\end{enumerate}
\end{theorem}

\begin{proof} 
The proof of (1) is a sequence of figures.
Item (2) follows from Figure \ref{ConwaySpheres2}. \end{proof}

\begin{figure}
    \centering
        \includegraphics[width=10cm]{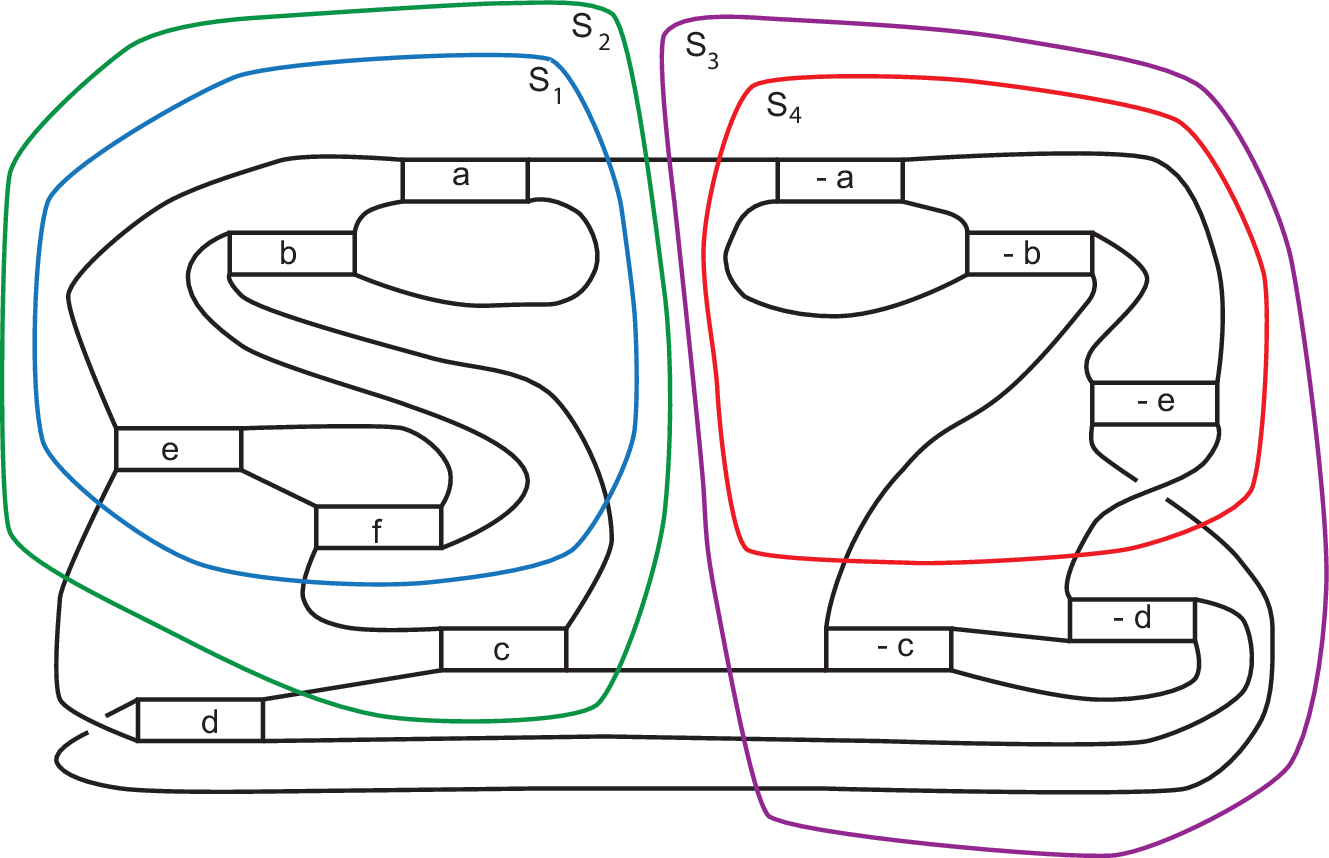}
        \caption{The knot $B(\ell,m,n,p,q)\cup R(1/0)$}
        \label{ConwaySpheres2}
        \end{figure}

Let ${\bar K}(a,b,c,d,e,f)$ be the covering knot associated to the tangle $Q(a,b,c,d,e,f)$.

\begin{theorem} 
\begin{enumerate}
\item $\bar K(a,b,c,d,e,f)$ is a hyperbolic knot. 

\item For a certain integral slope $s$, 
$\bar K(a,b,c,d,e,f)(s)$ contains four disjoint non-isotopic incompressible tori.

\item In fact  $\bar K(a,b,c,d,e,f)(s)$
is a graph manifold, determined by a linear graph consisting of 5 vertices and 4 edges, where the extremal vertices represent Seifert fiber space over a disk $D^2$ and two exceptional fibers, and the internal vertices represent Seifert fibered spaces over an annulus $A^2$ with one exceptional fiber. Furthermore the Seifert invariants of the pieces are given by $(D^2;{\frac{a}{ab-1},\frac{f}{ef-1}})$, $(A^2;{\frac{1}{c}})$, $({A^2; \frac{1}{d-1}})$, $(A^2;{\frac{c}{1-cd}})$, $(D^2;{\frac{a}{ab-1},\frac{e-1}{e}})$. 
\end{enumerate}
\end{theorem}

\begin{proof} That $\bar K$ is a hyperbolic knot, follows from a similar argument to that of Theorem \ref{main}. The rest of the proof follows from Figure \ref{ConwaySpheres2}. \end{proof}

\section{Properties of the examples}

The knots $\tilde K(\ell,m,n,p,q)$ and $\bar K(a,b,c,d,e,f)$ have some properties that can be deduced from Figures \ref{ovillo}, \ref{ConwaySpheres}, \ref{tangle} and \ref{ConwaySpheres2}. We explore two constructions.

Look at each of the Conway spheres $C_i$ of Figure \ref{ConwaySpheres} but as lying in $(S^3,U)=B(\ell,m,n,p,q)\cup R(0)$. Note that each sphere intersects $U$ in six points, determining then a $3$-string tangle decomposition of $U$.  Let $\tilde C_i$ be the lifting of $C_i$ in the double cover of $S^3$ branched along $U$, which is the $3$-sphere. Each $\tilde C_i$ is a genus two surface embedded in $S^3$, and it follows that the knot $\tilde K=\tilde K(\ell,m,n,p,q)$ lies over each of the surfaces $\tilde C_i$, it is a non-separating curve there, and the slope in $N(\tilde K)$ determined by the embedding in $\tilde C_i$, is precisely the exceptional slope $r$. Note that $\tilde C_1$ separates $U$ into a $3$-string trivial tangle and a non-trivial tangle. This implies that $\tilde C_1$ bounds a genus two handlebody $H_1$ in $S^3$, which is knotted, that is, the complement of $H_1$ is not a handlebody. In the same way we see that $\tilde C_2$ and $\tilde C_3$ bound knotted handlebodies $H_2$ and $H_3$, such that $H_1 \subset H_2 \subset H_3$. Analogously, $\tilde C_4$ and $\tilde C_5$ bound handlebodies $H_4$ and $H_5$ such that $H_5\subset H_4$. But note that $H_3$ and $H_4$ are disjoint, they form a kind of handlebody link.
Note also that $H_2$ is obtained from $H_1$ by an annulus interchange, that is, there is an annulus $A_1$ in the exterior of $H_1$, whose boundary lies in $\tilde C_1$, such that $(S_1 \backslash N(\partial A_1)) \cup (\partial N(A_1) \backslash N(\partial A_1))$ is the surface $\tilde C_2$. Similarly, $\tilde C_3$ is obtained from $\tilde C_2$ by an annulus interchange, and also $\tilde C_4$ is obtained from $\tilde C_3$ and $\tilde C_5$ is obtained from $\tilde C_4$ by an annulus interchange.

We can estimate the tunnel number of $\tilde K(\ell,m,n,p,q)$, $\mathrm{tn}(\tilde K)$. A tangle $(B,t)$ has a $k$-bridge decomposition if there is a sphere $S$ in $B$ meeting the strings of the tangle in $2k$ points, and decomposing $(B,t)$ into two parts, one being a $k$-string trivial tangle and the other a $k$-string trivial tangle relative to $\partial B$, that is, the tangle consists of $(\partial B \times I , \partial t \times I)$, union a collection of $2k-2$ arcs isotopic to disjoint arcs lying in $S$ (see Figure 9 in \cite{EM2}).
It is not so difficult to see that $B(\ell,m,n,p,q,r)$ admits a $5$-bridge decomposition. This implies that $\mathrm{tn}(\tilde K)\leq 3$. On the other hand $\tilde K(r)$ is an irreducible manifold containing five incompressible, disjoint, non parallel tori. By a result of Kobayashi \cite{Ko}, the Heegaard genus of $\tilde K(r)$ cannot be two, for manifolds with Heegaard genus two can have at most two incompressible disjoint tori. So $\tilde K$ cannot be a tunnel number one knot, then $2\leq \mathrm{tn}(\tilde K) \leq 3$.

A similar construction can be done for the knots $\bar K(a,b,c,d,e,f)$. 
Look at each of the Conway spheres $S_i$ of Figure \ref{ConwaySpheres2} but as lying in $(S^3,U)=Q(a,b,c,d,e,f)\cup R(0)$. Note that each sphere intersects $U$ in six points, determining then a $3$-string tangle decomposition of $U$.  Let $\bar S_i$ be the lifting of $S_i$ in the double cover of $S^3$ branched along $U$, which is the $3$-sphere. Each $\bar S_i$ is a genus two surface embedded in $S^3$, and it follows that the knot $\bar K=\bar K(a,b,c,d,e,f)$ lies over each of the surfaces $\bar S_i$, it is a non-separating curve there, and the slope in $N(\bar K)$ determined by the embedding in $\bar S_i$, is precisely the exceptional slope $s$. Note that $S_1$ separates {$U$} into a $3$-string trivial tangle and a non-trivial tangle. This implies that $\bar S_1$ bounds a genus two handlebody $\bar H_1$ in $S^3$, which is knotted, that is, the complement of $H_1$ is not a handlebody. In the same way we see that $\bar S_2$ bounds a knotted handlebody $\bar H_2$, and $\bar S_3$ bounds a handlebody $H_3$, such that $H_1 \subset H_2 \subset H_3$. But in this case $H_3$ is unknotted, that is, $\bar S_3$ bounds another handlebody $\bar H_4$, and $S_4$ bounds a knotted handlebody $H_5$, such that $H_5\subset H_4$. That is, $\bar K$ lies on the boundary of a genus two handlebody standardly embedded in $S^3$.
Again, $\mathrm{tn}(\bar K)$ cannot be one by \cite{Ko}, then $\mathrm{tn}(\bar K)=2$.

By taking some forbidden values of the parameter defining $\tilde K$ or $\bar K$, it is possible to find examples of hyperbolic knots having a surgery producing a graph manifod which have two or three incompressible tori. In particular the knots $\bar K(a,b,c,-2,-1,1)$ have a surgery producing a manifold with two incompressible tori.
It is not diffucult to show that the tangle $Q(a,b,c,d,e,f)$ has a 3-bridge decomposition, and then $\mathrm{tn}(\bar K(a,b,c,-2,-1,1))=1$. These knots admit the maximal number of tori that can appear by surgery in a tunnel number one knot.

The knots constructed in this paper are somehow related to the known examples of knots that admit several genus one Seifert surfaces \cite{EGR}, \cite{EGMR}, \cite{Valdez}. Look at the trivial knot $(S^3,U)= B(\ell,m,n,p,q)\cup R(0)$ considered before. Note that there is a trivial knot $L$ in its complement
which bounds 6 disjoint non-parallel disks $D_1$, $D_2$, $D_3$, $D_4$, $D_5$, $D_6$, each intersecting $K$ in 3 points, as shown in Figure \ref{seisdiscos}. Let $D_1$ be the orange disk bounded by $L$ as shown in Figure \ref{seisdiscos}.

\begin{figure}
    \centering
    \includegraphics[width=7cm]{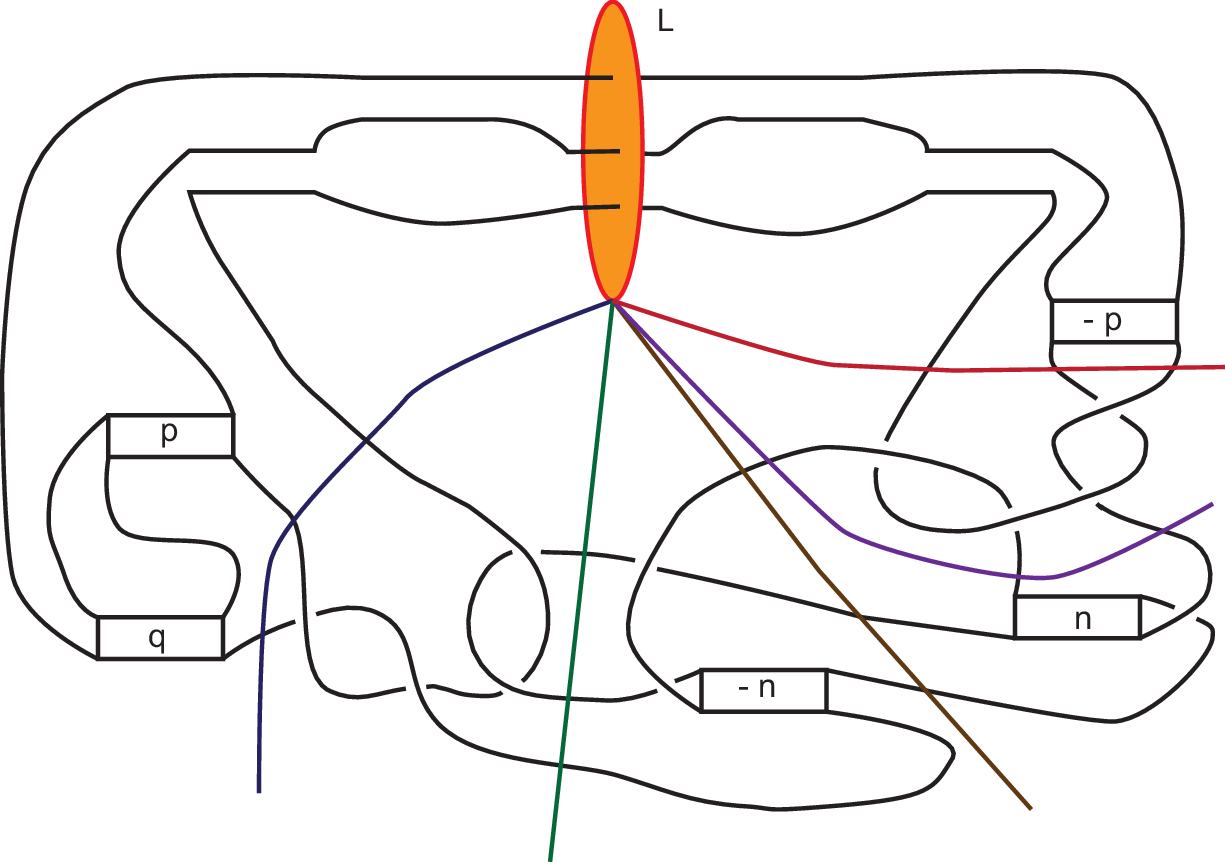}
    \caption{The link $U\cup L$}
    \label{seisdiscos}
    \end{figure}

As before, let $\pi\colon S^3 \rightarrow S^3$ be the double cover of $S^3$ branched along $U$, which is the 3-sphere because $U$ is the trivial knot. 
Let $\tilde L =\pi^{-1}(L)$ be the lift of the knot $L$, this is a knot in $S^3$, because the linking number between $U$ and $L$ is odd. Let $T_i = \pi^{-1}(D_i)$, for $i=1,2,3,4,5,6$. These are genus one Seifert surfaces bounded by $\tilde L$. That is, $\tilde L$ bounds $6$ disjoint, non-isotopic genus one Seifert surfaces. 
However $\tilde L$ is not a hyperbolic knot, but it is a satellite knot. This is because there is a Conway sphere in $U\cup L$ disjoint from $L$, and this sphere lifts to an incompressible torus in the exterior of $\tilde L$. Note that this Conway sphere intersects the disk $D_1$, and it is disjoint from the other disks. 

By using a variation of the tangle $B(\ell,m,n,p,q)$, hyperbolic knots having 5 disjoint genus one Seifert surfaces can be constructed, see \cite{EGMR}. Presumably, this family of knots coincides with the family of knots having five genus one Seifert surfaces constructed in \cite{Valdez}.

We note that the knot $\tilde K(\ell,m,n,p,q)$ can be recovered from $\tilde L$.
Note that the satellite torus of $\tilde L$ intersect the Seifert surface $T_1$ and it is disjoint from the other Seifert surfaces.
Take a complicated knot $K^*$ lying in $T_1$, which intersects the satellite torus. $T_1^*=T_1\backslash N(K^*)$ is a pair of pants, with two boundary components lying in $N(K^*)$, say with slope $r$. 
Do Dehn surgery along the slope $r$. After capping off $T_1^*$ with meridian disks of the surgered solid torus,
it becomes a disk, and the union of this disk and each of the $T_i$, $2 \leq i\leq 6$, is an incompressible torus in the surgered manifold.

\end{document}